\documentclass{amsart}
\usepackage[english]{babel}
\usepackage{iftex}
\usepackage[all,cmtip]{xy}
\usepackage{hyperref}
\usepackage{lmodern}
\usepackage[initials]{amsrefs}
\usepackage{graphicx}
\usepackage[english]{babel}
\usepackage{amssymb}
\usepackage{amsthm}
\usepackage{bm}
\usepackage{enumerate}
\usepackage{stmaryrd}
\usepackage{xcolor}

\newtheorem{theorem}{Theorem}[section]
\newtheorem{lemma}[theorem]{Lemma}
\newtheorem{proposition}[theorem]{Proposition}
\newtheorem{corollary}[theorem]{Corollary}

\theoremstyle{definition}

\newtheorem{example}[theorem]{Example}

\theoremstyle{remark}
\newtheorem{remark}[theorem]{Remark}
\newtheorem{notation}[theorem]{Notation}

\numberwithin{equation}{section}

\newcommand*{\Rmod}{R\text{-}\textbf{Mod}}
\newcommand*{\RCmod}{R\mathcal C\text{-}\textbf{Mod}}
\newcommand*{\RSmod}{R         S\text{-}\textbf{Mod}}
\newcommand*{\C}{\mathcal C} 
\newcommand*{\D}{\mathcal D} 

\DeclareMathOperator{\Img}{Im}
\DeclareMathOperator{\id}{id}
\DeclareMathOperator{\Ob}{Ob}
\DeclareMathOperator{\supp}{Supp}
\DeclareMathOperator{\ind}{ind}
\DeclareMathOperator{\res}{res}
\DeclareMathOperator{\Mor}{Mor}
\DeclareMathOperator{\colim}{colim}
\DeclareMathOperator{\mub}{mub}
\DeclareMathOperator{\mlb}{mlb}
\DeclareMathOperator{\zip}{zip}
\DeclareMathOperator{\unzip}{unzip}

\begin{document}

\title{Finite presentation of finitely determined modules}


\author{Eero Hyry}
\email{eero.hyry@tuni.fi}
\address{Faculty of Information Technology and Communication Sciences,
         Tampere University,
				 Kanslerinrinne 1 (Pinni B),
         Tampere, 33100,
				 Finland
				 }

\author{Markus Klemetti}
\email{markus.o.klemetti@gmail.com}
\address{Faculty of Information Technology and Communication Sciences,
         Tampere University,
				 Kanslerinrinne 1 (Pinni B),
         Tampere, 33100,
				 Finland
				 }
				

\date{January 19, 2023}

\begin{abstract}
In this article we study certain notions of `tameness' for the persistence modules studied in topological data analysis. In particular, we show that after adding infinitary points the so called finitely determined modules become finitely presented.
\end{abstract}

\maketitle



\section*{Introduction}
This article is motivated by topological data analysis, which is a recent field of mathematics studying the shape of data. One of the main methods of topological data analysis is persistent homology. In persistent homology one studies the data by associating a filtered topological space to it. By taking homology with coefficients in a field, one obtains a diagram of vector spaces and linear maps. This diagram is called a persistence module. In the standard case, the filtration is indexed by $\mathbb Z$ or $\mathbb R$, but the indexing
set can by any poset. Carlsson and Zomorodian realized that one can consider persistence modules indexed by $\mathbb Z^n$ as 
$\mathbb Z^n$-graded modules over a polynomial ring of $n$ variables (see~\cite{Carlsson}*{p.~78, Thm.~1}). This opened the way for methods of commutative algebra and algebraic geometry in topological data analysis. However, it is important to consider also more general indexing sets. More formally, a persistence module indexed by a poset $\C$ with coefficients in a field $k$ is a functor from $\C$, interpreted as a category, to the category of $k$-vector spaces. For the sake of generality, instead of a field $k$, we prefer in this article to work with any commutative ring $R$. Following the terminology of representation theory, we call a functor $\C\rightarrow \Rmod$ an $R\C$-module. In this terminology, a persistence module is then a $k\C$-vector space.

Persistence modules need not be finitely presented. For computational reasons, one has therefore introduced several notions of `tameness' for them. In this context, Miller defines in~\cite{Miller}*{p.~24, Def.~4.1} an encoding of an $R\C$-module $M$ by a poset $\D$ to be a poset morphism $f\colon \C\rightarrow \D$ with an $R\D$-module $N$ such that the restriction $\res_f N\cong M$. We define in \cite{Hyry}*{p.~22, Def.~4.1} an $R\C$-module $M$ to be \emph{$S$-determined} if there exists a subset $S\subseteq \C$ such that $\supp(M)\subseteq {\uparrow} S$, and for every $c\leq d$ in $\C$ the implication
\[
S\cap {\downarrow} c =  S \cap {\downarrow} d \ \Rightarrow \ M(c\leq d) \text{ is an isomorphism}
\]
holds. For any  $T\subseteq\C$, we use the usual notations 
\[
\hbox{${\uparrow} T := \{c\in \C \mid t\leq c \text{ for some }t\in T\}$}
\]
and
\[
\hbox{${\downarrow} T := \{c\in \C \mid c\leq t \text{ for some }t\in T\}$}
\]
for the upset generated and the downset cogenerated by $T$, respectively.
This is a straightforward generalization of the notion of a `positively $a$-determined' $\mathbb N^n$-graded module,
where $a\in \mathbb N^n$, as defined in \cite{Miller3}*{p.~186, Def.~2.1}. One can look at them as modules determined 
by their restriction to the interval $[0,a]\subseteq \mathbb N^n$. They are finitely generated.
Positively $a$-determined modules have been much studied by commutative algebraists. See e.g.~\cite{Brun} and the references therein. 

Suppose now that $S\subseteq \C$ is a finite set. We will consider the set $\tilde{S}$ of all minimal upper bounds of the subsets of $S$. We are going to define a functor $\alpha \colon \C\rightarrow \tilde{S}$ by mapping an element of $\C$ to the unique minimal upper bound of the elements of $S$ below it. In our main result, Theorem~\ref{S-det_luo_1}, we will prove that $M$ is $S$-determined for some finite $S\subseteq \C$ if and only if $\alpha$ is an encoding of $M$. 

The definition given by Miller in \cite{Miller3}*{p.~186, Def.~2.1} includes also the so called `finitely determined' modules. They are further studied in the context of topological data-analysis in~\cite{Miller}. Finitely determined modules are 
$\mathbb Z^n$-graded modules fully determined by their restriction to an interval $[a,b]\subseteq \mathbb Z^n$.They are not finitely presented in general. However, as a consequence of Theorem~\ref{S-det_luo_1} we can show in Theorem~\ref{aar_maar2} that after adding infinitary points to $\mathbb Z^n$, finitely determined modules in fact become finitely presented. We follow here an idea due to Perling (see~\cite{Perling}*{p.~16}). We also show in Proposition \ref{PerlingProposition} that our terminology is compatible with that of admissible posets used in~\cite{Perling}.



\section{Preliminaries}
Throughout this article we use the terminology of category theory. We will always assume that $\C$ is a small category and $R$ a commutative ring. For any set $X$, we denote by $R[X]$ the free $R$-module generated by $X$. An $R\C$-module is a functor from $\C$ to the category of $R$-modules. A morphism between $R\C$-modules is a natural transformation. For more details on $R\C$-modules, we refer to~\cite{Luck} and~\cite{TomDieck}.

Recall first that an $R\C$-module $M$ is called
\begin{itemize}
\item \emph{finitely generated} if there exists an epimorphism
\[
\bigoplus_{i\in I} R[\Mor_{\C}(c_i,-)] \rightarrow M,
\]
where $I$ is a finite set, and $c_i\in \C$ for all $i\in I$;
\item \emph{finitely presented} if there exists an exact sequence
\[
\bigoplus_{j\in J} R[\Mor_{\C}(d_j,-)] \rightarrow \bigoplus_{i\in I} R[\Mor_{\C}(c_i,-)] \rightarrow M\rightarrow 0,
\]
where $I$ and $J$ are finite sets, and $c_i, d_j\in \C$ for all $i\in I$ and $j\in J$. 
\end{itemize}
See, for example,~\cite{Popescu}.

Let $\varphi\colon S\rightarrow \C$ be a functor between small categories. Recall that the \emph{restriction} $\res_{\varphi}\colon \RCmod \rightarrow \RSmod$ is the functor defined by precomposition with $\varphi$, and the \emph{induction} $\ind_{\varphi}\colon \RSmod \rightarrow \RCmod$ is its left Kan extension along $\varphi$. The induction is the left adjoint of the restriction. The counit of this adjunction gives us for every $R\C$-module $M$ the  \emph{canonical morphism} 
\[
\mu_M\colon  \ind_{\varphi}\res_{\varphi} M\rightarrow M.
\]

More explicitly, for any $R\C$-module $M$ and $RS$-module $N$, we have the pointwise formulas 
\[
(\res_{\varphi} M)(s) = M(\varphi(s)) \quad \text{and}\quad (\ind_{\varphi} N)(c)  = \underset{(t,u)\in (\varphi/c)}\colim N(t)
\]
for all $s\in S$ and $c\in \C$. Here $(\varphi/c)$ denotes the slice category. Its objects are pairs $(s,u)$, where $s\in S$ and $u\colon \varphi(s)\rightarrow c$ is a morphism in $\C$. For $(s,u),(t,v)\in \Ob
(\varphi/c)$, a morphism $(s,u) \rightarrow (t,v)$ is a morphism $f\colon s\rightarrow t$ in $S$ with $v\varphi(f)=u$.
We will typically assume that $S$ is a full subcategory of $\C$ and that $\varphi$ is the inclusion functor. In this case, we use the notations $\res_S$ and $\ind_S$ instead of $\res_{\varphi}$ and $\ind_{\varphi}$. If $\C$ is also a poset, the latter formula yields
\[
(\ind_S N)(c) = \underset{t\in S, \ t\leq c}\colim N(t).
\] 

Let $\C$ be a small category and $S\subseteq \C$ a full subcategory. An $R\C$-module $M$ is said to be \emph{$S$-generated} if the natural morphism 
\[
\rho_M\colon \bigoplus_{s\in S}M(s)[\Mor_\C(s,-)]\rightarrow M
\]
is an epimorphism. Here
\[
M(s)[\Mor_\C(s,-)]:= M(s) \otimes_R R[\Mor_\C(s,-)],
\]
where the tensor product is taken pointwise. Since the morphism $\rho_M$ factors through the canonical morphism $\mu_M$, we see that $M$ is $S$-generated if and only if $\mu_M$ is an epimorphism.

Following \cite{Djament}*{p.~13, Prop.~2.14}, we say that $M$ is \emph{$S$-presented} if it is $S$-generated and the following condition holds: Given an exact sequence of $R\C$-modules
\[
0\rightarrow L \rightarrow N \rightarrow M \rightarrow 0,
\]
where $N$ is $S$-generated, then $L$ is $S$-generated. It is shown in \cite{Djament}*{p.~13, Prop.~2.14}, that $M$ is $S$-presented if and only if $\mu_M$ is an isomorphism.


\section{Modules over strongly bounded posets}

In order to prove Theorem~\ref{S-det_luo_1}, we need to recall some order theory. In the following, $\C$ always denotes a poset.

\begin{notation}\label{hattu}
Let $S\subseteq \C$ be a finite subset.  We denote the set of minimal upper bounds of $S$ by $\mub(S)$. If $S$ is finite, we set
\[
\hat{S} := \bigcup_{\emptyset\neq S'\subseteq S} \mub(S').
\]
In other words, $\hat{S}$ is the set of minimal upper bounds of non-empty subsets of $S$.
\end{notation}

We say that the poset $\C$ is \emph{strongly bounded from above} if every finite $S\subseteq \C$ has a unique minimal upper bound in $\C$. If $\C$ is strongly bounded from above, then $\hat{S}$ is finite. The condition of $\C$ being strongly bounded from above is equivalent to $\C$ being a bounded join-semilattice. Also note that if $\C$ is strongly bounded from above, then $\C$ is weakly bounded from above and mub-complete, as defined in \cite{Hyry}*{p.~23, Def.~4.5, Def.~4.6}. 

Let $\C$ be strongly bounded from above, and let $S\subseteq \C$ be a finite set. From now on, we consider $\mub(S)$ as an element of $\C$, and not as a (one element) set. In particular, every element of $\hat{S}$ is then of the form $\mub(S')$, where $S'\subseteq S$ is a non-empty subset. Viewing $\C$ as a join-semilattice, we have the join-operation
\[
a\vee b := \mub(a,b):=\mub(\{a,b\}).
\]
Extending this operation to finite sets, we get an operation that coincides with taking minimal upper bounds.

\begin{lemma}\label{hattu_pois}
Let $\C$ be strongly bounded from above, and let $S\subseteq \C$ be a finite subset. Then $\hat{\hat S} = \hat{S}$.
\end{lemma}

\begin{proof}
An element $s\in \hat{\hat{S}}$ may be written as
\[
s=\mub(\mub(S_1),\ldots,\mub(S_n)), 
\]
where $S_1,\ldots,S_n$ are (finite) non-empty subsets of $S$. Since the join-operation is associative in join-semilattices, we see that
\[
s=\bigvee_{i=1}^n(\bigvee S_i) = \bigvee (\bigcup_{i=1}^n S_i).
\]
This implies that $s=\mub(S_1\cup\cdots\cup S_n)$, which belongs to $\hat{S}$ by definition.
\end{proof}

Assume that $\C$ is strongly bounded from above. Then $\C$ has a minimum element $\min(\C)=\mub(\emptyset)$. Let $S\subseteq \C$ be a finite subset. Denote 
\[
\tilde{S} := \hat{S}\cup\{\min(\C)\}.
\]
We define a poset morphism $\alpha_S\colon \C\rightarrow \tilde{S}$ by setting 
\[
\alpha_S(c)=\mub(S \cap{\downarrow} c)
\]
for every $c\in \C$. In other words, $\alpha_S$ maps each $c\in \C$ to the minimal upper bound of the elements of $S$ below it. To show that $\alpha_S$ actually is a poset morphism, suppose that $c\leq d$ in $\C$. Then $S\cap{\downarrow}c\subseteq S\cap {\downarrow} d$, which implies that $\alpha_S(c)\leq \alpha_S(d)$. 

\begin{proposition}\label{alfat}
Let $\C$ be strongly bounded from above, and let $S\subseteq \C$ be a finite subset. Then $\alpha_S=\alpha_{\hat{S}}=\alpha_{\tilde{S}}$.
\end{proposition}

\begin{proof}
Using Lemma \ref{hattu_pois}, we first note that $\tilde{\hat{S}}=\tilde{S}$ and $\tilde{\tilde{S}}=\tilde{S}$. Let $c\in \C$. We claim that
\[
\mub(S\cap {\downarrow}c) = \mub(\hat{S}\cap {\downarrow}c) = \mub(\tilde{S}\cap {\downarrow}c). 
\]
The latter equation follows from the fact that for all subsets $T\subseteq \C$, we have $\mub(T)=\mub(T\cup\{\min(\C)\})$. In particular, $\mub(T)=\min(\C)$, if $T=\emptyset$.

For the first equation, since $S\subseteq \hat{S}$, we have $\mub(S\cap {\downarrow}c) \leq \mub(\hat{S}\cap {\downarrow}c)$. On the other hand, $\hat{S}\cap{\downarrow}c$ is a subset of $\hat{S}$. Thus $\mub(\hat{S}\cap{\downarrow}c)\in \hat{\hat S}=\hat{S}$, where the equation follows from Lemma \ref{hattu_pois}. By the definition of $\hat{S}$, we may now write
\[
\mub(\hat{S}\cap{\downarrow}c)=\mub(s_1,\ldots,s_n),
\]
where $s_1,\ldots,s_n\in S$. Furthermore, $\mub(\hat{S}\cap{\downarrow}c)\leq c$, so we also have $s_1,\ldots,s_n\leq c$. This implies that
\[
\mub(s_1,\ldots,s_n) \leq \mub(S\cap {\downarrow}c),
\]
which completes the proof.
\end{proof}

Encouraged by Proposition \ref{alfat}, we will just write $\alpha$ instead of $\alpha_S$, if there is no risk of confusion. Before moving on to the main theorem of this section, we require one more lemma.

\begin{lemma}\label{hattu_alajoukko}
Let $\C$ be strongly bounded from above, and let $S\subseteq \C$ be a finite subset. Then $\hat{S}\cap {\downarrow}\alpha(c) = \hat{S}\cap {\downarrow}c$ for all $c\in\C$.
\end{lemma}

\begin{proof}
Let $c\in \C$. We immediately see that $\hat{S}\cap {\downarrow} \alpha(c) \subseteq \hat{S}\cap {\downarrow} c$, because $\alpha(c)\leq c$. Suppose that $d\in \hat{S}\cap {\downarrow} c$. We need to show that $d\leq \alpha(c)$. This follows from Proposition \ref{alfat}, because now
\[
\alpha(c)=\alpha_{\hat{S}}(c)=\mub(\hat{S}\cap{\downarrow}c).
\]
\end{proof}

Let $\C$ be strongly bounded from above, let $M$ be an $R\C$-module, and let $S\subseteq \C$ be a finite subset. The morphism $\alpha$ gives rise to a natural transformation
\[
T_{\alpha}\colon  \res_{\alpha}\res_{\tilde{S}} M \rightarrow M,
\] 
where for any $c\in \C$, $T_{\alpha,c}$ is the morphism
\[
M(\alpha(c)\leq c)\colon  (\res_{\alpha}\res_{\tilde{S}} M)(c)=M(\alpha(c)) \rightarrow M(c).
\]

We are now able to prove our main result

\begin{theorem}\label{S-det_luo_1}
Let $\C$ be strongly bounded from above, and let $M$ be an $R\C$-module. Given a finite subset $S\subseteq \C$, the following conditions are equivalent:
\begin{itemize}
\item[1)] For all $c\leq d$ in $\C$,
\[
S\cap {\downarrow} c =  S \cap {\downarrow} d \ \Rightarrow \ M(c\leq d) \text{ is an isomorphism};
\]
\item[2)] $T_\alpha\colon \res_{\alpha}\res_{\tilde{S}}M \rightarrow M$ is an isomorphism;
\item[3)] $\alpha$ is an encoding of $M$.
\end{itemize}
If $\min(\C)\in S$, then condition 1) says that $M$ is $S$-determined.
\end{theorem}

\begin{proof}
Suppose first that 1) holds. We can safely assume that $S$ includes the minimum element of $\C$, so that $\supp(M)\subseteq {\uparrow}S=\C$. This will not affect the sets $\hat{S}$ or $\tilde{S}$, nor the functor $\alpha$. Therefore $M$ is $S$-determined. We have proved in \cite{Hyry}*{p.~25, Cor.~4.13} that an $S$-determined module is $\hat{\hat{S}}$-presented. Lemma \ref{hattu_pois} now tells us that $M$ is $\hat{S}$-presented. So $M\cong \ind_{\hat{S}}\res_{\hat{S}}M$. This implies that for $c\in \C$,
\[
(\res_{\alpha}\res_{\tilde{S}}M)(c) = M(\alpha(c)) \cong \underset{d\leq \alpha(c), \ d\in \hat{S}}\colim M(d).
\]
Furthermore, by Lemma \ref{hattu_alajoukko}, we get
\[
\underset{d\leq \alpha(c), \ d\in \hat{S}}\colim M(d) = \underset{d\leq c, \ d\in \hat{S}}\colim M(d) = (\ind_{\hat{S}} \res_{\hat{S}} M)(c) \cong M(c).
\]

If 2) holds, we immediately see that the functor $\alpha$ with the $R\tilde{S}$-module $\res_{\tilde{S}}M$ is an encoding of $M$.

Finally, suppose that 3) is true. Assume that $c\leq d$ and $S\cap{\downarrow}c=S\cap{\downarrow}d$. We need to show that $M(c\leq d)$ is an isomorphism. Since $\alpha$ is an encoding of $M$, there exists an $R\tilde{S}$-module $N$ such that $\res_{\alpha}N\cong M$. Here $\res_{\alpha}N(c\leq d)$ is the morphism $N(\alpha(c)\leq \alpha(d))$. We note that
\[
\alpha(c)=\mub(S\cap{\downarrow}c)=\mub(S\cap{\downarrow}d)=\alpha(d),
\]
so the morphism $\res_{\alpha}N(c\leq d)$ is an isomorphism. Thus $M(c\leq d)$ is an isomorphism. Therefore 1) holds true.
\end{proof}


\section{Adding infinitary points}\label{fin_det_mod}

One approach to understand $R\mathbb Z^n$-modules better is to expand the set $\mathbb Z^n$ to include points at infinity. This idea has been utilized by Perling in \cite{Perling}. Set $\overline{\mathbb Z}:= \mathbb Z \cup \{-\infty\}$. It is easy to see that $\overline{\mathbb Z}^n$ inherits a poset structure from $\mathbb Z^n$. Any $R\mathbb Z^n$-module $M$ can be naturally extended to an $R \overline{\mathbb Z}^n$-module $\overline{M}$ by setting 
\[
\overline{M}(c) = \lim_{d\geq c, \ d\in \mathbb Z^n}M(d)
\]
 for all $c\in \overline{\mathbb Z}^n$. More formally, this is the coinduction of $M$ with respect to the inclusion $\mathbb Z^n\rightarrow \overline{\mathbb Z}^n$. The functor $M\mapsto \overline{M}$ establishes an equivalence of categories between the category $R\mathbb Z^n\text{-}\textbf{Mod}$ and its essential image in $R\overline{\mathbb Z}^n\text{-}\textbf{Mod}$.

Let $S\subseteq \overline{\mathbb Z}^n$ be a finite non-empty subset. We denote by $\mlb(S)$ the (unique) maximal lower bound of $S$. In this section, we will define a morphism $\beta$ ``dual'' to $\alpha$. The idea is to map an element to the maximal lower bound of the elements of $S$ above it. The morphism $\beta$ will play a crucial role in the proof of our main result, Theorem \ref{aar_maar2}. 

We restrict ourselves to \emph{cartesian} subsets of $\overline{\mathbb Z}^n$, i.e.~subsets of the form $S=S_1\times\cdots\times S_n$, where $S_1,\ldots,S_n$ are subsets of $\overline{\mathbb Z}$. In this situation, we can calculate $\alpha$ and $\beta$ coordinatewise. We begin with the following observation.

\begin{proposition}\label{karteesi_mub}
Let $p_i\colon \overline{\mathbb Z}^n\rightarrow \overline{\mathbb Z}$ be the canonical projection for every $i\in \{1,\ldots,n\}$, and let $S\subseteq \overline{\mathbb Z}^n$ be a finite non-empty subset. Then
\begin{itemize}
\item[1)] $\mub(S)=(\max(p_1(S)),\ldots,\max(p_n(S)))$;
\item[2)] $\mlb(S)=(\min(p_1(S)),\ldots,\min(p_n(S)))$.
\end{itemize}
\end{proposition}

\begin{proof}
Since both 1) and 2) are proved in the same way, we will only present the proof of 1) here. Let $i\in \{1,\ldots,n\}$. The existence of $\max(p_i(S))$ follows from the fact that $p_i(S)$ is non-empty, linearly ordered and finite. Write 
\[
d=(d_1,\ldots,d_n):=\mub(S).
\]
We will show that $d_i=\max(p_i(S))$. First, since $d$ is an upper bound of $S$ and the canonical projection $p_i$ preserves order, we see that $d_i=p_i(d)\geq \max(p_i(S))$.  Secondly, if $\max(p_i(S))<d_i$, then 
\[
d':=(d_1,\ldots,d_{i-1},\max(p_i(S)),d_{i+1},\ldots,d_n)
\] 
is an upper bound of $S$ such that $d'<d$, contradicting the minimality of $d$. Thus $d_i=\max(p_i(S))$.
\end{proof}

Let $S:=S_1\times\cdots\times S_n\subseteq \overline{\mathbb Z}^n$ be a cartesian subset. We write
\[
\overline{S}=\tilde{S_1}\times\cdots\times \tilde{S_n},
\]
where $\widetilde{S_i}=S_i\cup\{-\infty\}\subseteq \overline{\mathbb Z}$ for all $i\in \{1,\ldots,n\}$. Note that if $S$ is finite, then so is $\overline{S}$.

\begin{example}\label{laajennettu_vali}
Let $a\leq b$ in $\mathbb Z^n$. We write $a=(a_1,\ldots,a_n)$ and $b=(b_1,\ldots,b_n)$. For the closed interval
\begin{align*}
[a,b]&=\{c\in \mathbb Z^n\mid a\leq c\leq b\}=[a_1,b_1]\times\cdots\times [a_n,b_n],
\end{align*}
we have
\begin{align*}
\overline{[a,b]}&=\widetilde{[a_1,b_1]}\times\cdots\times\widetilde{[a_n,b_n]}\\
&=\{(c_1,\ldots,c_n)\mid a_i\leq c_i\leq b_i \ \text{or} \ c_i=-\infty\, \ (i\in \{1,\ldots,n\})\}.
\end{align*}
\end{example}

We now have

\begin{lemma}\label{aar_karteesi}
Let $S:=S_1\times\cdots\times S_n\subseteq \overline{\mathbb Z}^n$ be a finite cartesian subset, and let $T\subseteq S$ be a finite non-empty subset. Then
\begin{itemize}
\item[1)] $\mub(T)\in S$;
\item[2)] $\mlb(T)\in S$;
\item[3)] $\tilde{\overline{S}}=\overline{S}$.
\end{itemize}
\end{lemma}

\begin{proof}
To prove 1), let $p_i$ be the canonical projection $\overline{\mathbb Z}^n\rightarrow \overline{\mathbb Z}$ for all $i\in \{1,\ldots,n\}$. From Proposition \ref{karteesi_mub} 1), we get that
\[
\mub(T)=(\max(p_1(T)),\ldots,\max(p_n(T))).
\]
Thus $\mub(T)\in S$, because $p_i(T)\subseteq p_i(S)=S_i$ for all $i\in \{1,\ldots,n\}$.

Next, the proof for 2) is done in the same way as 1), this time using Proposition \ref{karteesi_mub} 2).

Finally, for 3), we note that $\overline{S}$ is finite and cartesian, so 1) implies $\hat{\overline{S}}=\overline{S}$. Since $\overline{S}$ already contains the minimum element of $\overline{\mathbb Z}^n$, we get
\[
\tilde{\overline{S}}=\hat{\overline{S}}\cup\{(-\infty,\ldots,-\infty)\}
=\overline{S}\cup\{(-\infty,\ldots,-\infty)\}
=\overline{S}.
\]
\end{proof}

Let $S:=S_1\times\cdots\times S_n\subseteq \overline{\mathbb Z}^n$ be a finite cartesian subset. Since $\tilde{\overline{S}}=\overline{S}$ by Lemma \ref{aar_karteesi} 3), we have a poset morphism $\alpha:=\alpha_{\overline{S}}\colon \overline{\mathbb Z}^n\rightarrow \overline{S}$, where
\[
\alpha(c)= \mub(\overline{S}\cap {\downarrow} c)
\]
for all $c\in \overline{\mathbb Z}^n$. By Lemma \ref{aar_karteesi} 2), we now can define a ``dual'' poset morphism $\beta:=\beta_{S}\colon  \overline{S}\rightarrow S$ by setting
\[
\beta(c)=\mlb(S \cap {\uparrow} c)
\]
for all $c\in \overline{S}$. Here the set $S\cap{\uparrow}c$ is always non-empty, because $S$ is final in $\overline{S}$.

We can now give coordinatewise formulas for $\alpha$ and $\beta$.

\begin{proposition}\label{alfabeta_lemma}
We write $\alpha_i:=\alpha_{\overline{S_i}}$ and $\beta_i:=\beta_{S_i}$ for all $i\in\{1,\ldots,n\}$. For $c:=(c_1,\ldots,c_n)\in \overline{\mathbb Z}^n$, we have
\begin{itemize}
\item[1)] $\alpha(c)=(\alpha_1(c_1),\ldots, \alpha_n(c_n))$;
\item[2)] if $c \in \overline{S}$, then $\beta(c)=(\beta_1(c_1),\ldots,\beta_n(c_n))$.
\end{itemize}
\end{proposition}

\begin{proof}
To prove 1), we will first show that 
\[
p_i(\overline{S}\cap{\downarrow}c)=\overline{S_i}\cap{\downarrow}c_i,
\]
 where $p_i\colon \overline{\mathbb Z}^n\rightarrow \overline{\mathbb Z}$ is the canonical projection for all $i\in\{1,\ldots,n\}$. Since $p_i(\overline{S})=\overline{S_i}$ and $p_i({\downarrow}c)={\downarrow}c_i$, we see that $p_i(\overline{S}\cap{\downarrow}c)\subseteq \overline{S_i}\cap{\downarrow}c_i$. For the other direction, suppose that $d\in \overline{S_i}\cap {\downarrow}c_i$. Then $d\leq c_i$, so we have an element
\[
d':=(-\infty,\ldots,-\infty,d,-\infty,\ldots,-\infty) \in \overline{S}\cap{\downarrow}c
\]
such that $p_i(d')=d$. Hence $p_i(\overline{S}\cap{\downarrow}c)=\overline{S_i}\cap{\downarrow}c_i$. Now, using this result and Proposition \ref{karteesi_mub} 1), we get
\begin{align*}
\alpha(c) &= \mub(\overline{S}\cap{\downarrow} c)\\
&=(\max(\overline{S_1}\cap{\downarrow}c_1),\ldots,\max(\overline{S_n}\cap{\downarrow}c_n))\\
&=(\alpha_1(c_1),\ldots,\alpha_n(c_n)).
\end{align*}

For 2), the proof is similar. Let $c\in \overline{S}$. We will first show that 
\[
p_i(S\cap{\uparrow}c)= S_i\cap{\uparrow}c.
\]
From $p_i(S)=S_i$ and $p_i({\uparrow}c)={\uparrow}c_i$, we see that $p_i(S\cap{\uparrow}c)\subseteq S_i\cap{\uparrow}c_i$. Next, suppose that $d\in S_i\cap{\uparrow}c_i$. Since $c\in \overline{S}$, there is an element $s:=(s_1,\ldots,s_n)\in S$ such that $s\geq c$. Because $d\geq c_i$ and $S$ is cartesian, we again have an element
\[
d':=(s_1,\ldots,s_{i-1},d,s_{i+1},\ldots,s_n)\in S\cap{\uparrow}c
\] 
such that $p_i(d')=d$. Thus $p_i(S\cap{\uparrow}c)= S_i\cap{\uparrow}c$. To finish the proof, we use Proposition \ref{karteesi_mub} 2):
\begin{align*}
\beta(c) &= \mlb(S\cap{\uparrow} c)\\
&=(\min(S_1\cap{\uparrow}c_1),\ldots,\min(S_n\cap{\uparrow}c_n))\\
&=(\beta_1(c_1),\ldots,\beta_n(c_n)).
\end{align*}
\end{proof}

We note that $\alpha$ and $\beta\circ\alpha$ are ``continuous'' in the following sense.

\begin{proposition}\label{jatkuvuus}
Let $c:=(c_1,\ldots,c_n)\in \overline{\mathbb Z}^n$.
\begin{itemize}
\item[1)] If $N$ is an $R\overline{S}$-module, then 
 \[
\lim_{d\geq c, \ d\in \mathbb Z^n}N(\alpha(d)) \cong N(\alpha(c)).
\]
\item[2)] If $Q$ is an $RS$-module, then 
\[
\lim_{d\geq c, \ d\in \mathbb Z^n}Q((\beta\circ\alpha)(d)) \cong Q((\beta\circ\alpha)(c)).
\]
\end{itemize}
\end{proposition}

\begin{proof}
For 1), suppose that $N$ is an $R\overline{S}$-module. Let $c':=(c'_1,\ldots,c'_n)\in \overline{\mathbb Z}^n$ as follows: For any $i\in \{1,\ldots, n\}$, we set $a_i=\min(S_i\cap \mathbb Z)$, if it exists, and
\[
c'_i:=\left\{\begin{aligned}&\max(c_i,0), \ \text{if} \ S_i\cap \mathbb Z=\emptyset;\\
&\max(c_i, a_i-1), \ \text{otherwise}.
\end{aligned}\right.
\]
This guarantees that we always have $c\leq c'$ and $c'\in \mathbb Z^n$. With the notation from Proposition \ref{alfabeta_lemma}, we may write 
\[
\alpha(c) = (\alpha_1(c_1),\ldots,\alpha_n(c_n)).
\]
 Let $i\in \{1,\ldots,n\}$. If $S_i\cap \mathbb Z=\emptyset$, then $\alpha_i(c'_i)=-\infty=\alpha_i(c_i)$. Similarly, if $c_i'=a_i-1$, then $\alpha_i(c_i')=-\infty=\alpha_i(c_i)$. Thus $\alpha(c)=\alpha(c')$ in all cases. Since $\alpha$ is a poset morphism, we see that for all $d\in \mathbb Z^n$ such that $c\leq d\leq c'$,
\[
\alpha(c)=\alpha(d)= \alpha(c'),
\]
and therefore
\[
N(\alpha(c))=N(\alpha(d))=N(\alpha(c')).
\]
Furthermore, because the set $\{d\in \mathbb Z^n \mid c\leq d\leq c'\}$ is an initial subset of the set $\{d\in \mathbb Z^n\mid c\leq d\}$, we have
\[
\lim_{d\geq c, \ d\in \mathbb Z^n}N(\alpha(d))\cong \lim_{c\leq d\leq c', \ d\in \mathbb Z^n}N(\alpha(d)) \cong N(\alpha(c)).
\]

Next, for 2), let $Q$ be an $RS$-module. Now $\res_{\beta}Q$ is an $R\overline{S}$-module, so by 1), we have
\[
\lim_{d\geq c, \ d\in \mathbb Z^n}(\res_{\beta}Q)(\alpha(d))\cong (\res_{\beta}Q)(\alpha(c)).
\]
On the other hand, by definition, for all $e\in \overline{\mathbb Z}^n$,
\[
(\res_{\beta}Q)(\alpha(e))= Q(\beta(\alpha(e)))= Q((\beta\circ\alpha)(e)).
\]
This means that we may write the above isomorphism as
\[
\lim_{d\geq c, \ d\in \mathbb Z^n}Q((\beta\circ\alpha)(d))\cong Q((\beta\circ\alpha)(c)).
\]
\end{proof}

\begin{corollary}\label{jatkuvuus_koro}
Let $N$ be an $R\overline{\mathbb Z}^n$-module, and let $c\in \overline{\mathbb{Z}}^n$. Then
\begin{itemize}
\item[1)] $\displaystyle \lim_{d\geq c, \ d\in \mathbb Z^n}N(\alpha(d)) \cong N(\alpha(c))$;
\item[2)] $\displaystyle \lim_{d\geq c, \ d\in \mathbb Z^n}N((\beta\circ\alpha)(d)) \cong N((\beta\circ\alpha)(c))$.
\end{itemize}
\end{corollary}

\begin{proof}
For 1), we note that $\res_{\overline{S}}N$ is an $R\overline{S}$-module, where $(\res_{\overline{S}}N)(d)=N(d)$ for all $d\in\overline{S}$. We may then apply Proposition \ref{jatkuvuus} 1) to get the result. For 2), we use Proposition \ref{jatkuvuus} 2) on the $RS$-module $\res_{S}N$.
\end{proof}


\section{Finitely determined modules}

Let $M$ be an $R\C$-module. We say that $M$ is \emph{pointwise finitely presented} if $M(c)$ is finitely presented for all $c\in \C$. Slightly generalizing the definition of Miller in \cite{Miller}*{p.~25, Ex.~4.5}, where $R=k$ is a field, we say that an $R\mathbb Z^n$-module $M$ is \emph{finitely determined}, if $M$ is pointwise finitely presented, and for some $a\leq b$ in $\mathbb Z^n$, the convex projection $\pi\colon \mathbb Z^n\rightarrow [a,b]$ gives $M$ an encoding by the closed interval $[a,b]\subseteq \mathbb Z^n$. Here the convex projection $\pi$ takes every point in $\mathbb Z^n$ to its closest point in the interval $[a,b]$. If $a=(a_1,\ldots,a_n)$ and $b=(b_1,\ldots,b_n)$, we have for any $c:=(c_1,\ldots,c_n)\in \mathbb Z^n$,
\[
\pi(c) = (\pi_1(c_1),\ldots,\pi_n(c_n)),
\]
where
\[
\pi_i(c_i)=\max(a_i,\min(c_i,b_i))
\]
for all $i\in\{1,\ldots,n\}$. Note that a pointwise finitely presented $R\mathbb Z^n$-module $M$ is finitely determined if and only if there exists a closed interval $[a,b]\subseteq \mathbb Z^n$ such that the morphisms $M(c\leq c+e_i)$ ($i=1,\ldots,n$) are isomorphisms whenever $c_i$ lies outside $[a_i,b_i]$.

\begin{remark}\label{aar_maar_kuvailu}
Let $M$ be an $R\mathbb Z^n$-module. Then $M$ is encoded by the closed interval $[a,b]$ with the convex projection $\pi\colon \mathbb Z^n\rightarrow [a,b]$ if and only if $M\cong \res_{\pi}\res_{[a,b]} M$. Indeed, if $M\cong \res_{\pi}N$ for some $R[a,b]$-module $N$, then for all $c\in \mathbb Z^n$, we have
\[
M(c)\cong (\res_{\pi}N)(c) = N(\pi(c))=N(\pi(\pi(c)))\cong M(\pi(c))
\]
because for all $c\in \mathbb Z^n$, $\pi(\pi(c))=\pi(c)$.
\end{remark}

We would now like to investigate how the notion of finite determinacy relates to our notion of $S$-determinacy, when $S$ is finite and $M$ is pointwise finitely presented. While the requirement that $\supp(M)\subseteq {\uparrow}S$ does not necessarily hold for finitely determined modules, we do have the following:

\begin{proposition}\label{aar_maar}
Let $M$ be an $R\mathbb Z^n$-module, and $a,b\in \mathbb Z^n$ such that $a\leq b$. Set $u:=(1,1,\ldots,1)\in \mathbb Z^n$.
If $M$ is $[a+u,b]$-determined, then $M$ has an encoding by the closed interval $[a,b]$ with the convex projection $\pi\colon \mathbb Z^n\rightarrow [a,b]$. The converse implication holds if $\supp(M)\subseteq {\uparrow}a$.
\end{proposition}

\begin{proof}
For the first implication, suppose that $M$ is $[a+u,b]$-determined. We write $a=(a_1,\ldots,a_n)$ and $b=(b_1,\ldots,b_n)$. Let $c:=(c_1,\ldots,c_n)\in \mathbb Z^n$.  We note that if $c_i\leq a_i$ for some $i\in \{1,\ldots,n\}$, then also $\pi_i(c_i)\leq a_i$, so that $c,\pi(c)\notin \supp(M)$. Otherwise $c> a$, in which case $\pi(c)\leq c$ and $[a+u,b]\cap {\downarrow}\pi(c)=[a+u,b]\cap {\downarrow} c$.
Thus $M(\pi(c))\rightarrow M(c)$ is an isomorphism by the definition of $[a+u,b]$-determined modules, and $M\cong \res_{\pi}\res_{[a,b]}M$.

To prove the converse, assume that $\supp(M)\subseteq {\uparrow}a$ and $M$ has an encoding by the closed interval $[a,b]$ with the encoding convex projection $\pi\colon \mathbb Z^n\rightarrow [a,b]$. Let $c:=(c_1,\ldots,c_n)\in \mathbb Z^n$. Suppose that $c_i< a_i$ for some $i\in\{1,\ldots,n\}$. From the condition $\supp(M)\subseteq {\uparrow}a$, we see that $M(c)=0$. Since $M$ is finitely determined, we also have $M(\pi(c))=M(c)=0$. Thus $M(c)=0$ if $c_i\leq a_i$ for some $i\in \{1,\ldots,n\}$. If this is not the case, we have $c\geq a+u$. Let $c\leq d$ in $\C$ such that $a+u \leq c \leq d$ and $[a+u,b]\cap{\downarrow c}=[a+u,b]\cap{\downarrow d}$. This implies that $\pi(c)=\pi(d)$, so $M(c\leq d)$ is an isomorphism.
\end{proof}

To proceed, we have to shift our focus to $R\overline{\mathbb Z}^n$-modules. Let $a\leq b$ in $\mathbb Z^n$. With the notation from section \ref{fin_det_mod}, we will view the case $S=[a,b]$. In particular, we have $\alpha=\alpha_{\overline{[a,b]}}$ and $\beta=\beta_{[a,b]}$. Proposition \ref{alfabeta_lemma} gives us formulas for $\alpha$ and $\beta$. If $c:=(c_1,\ldots,c_n)\in \overline{\mathbb Z}^n$ and $d:=(d_1,\ldots,d_n)\in \overline{[a,b]}$, then
\[
\alpha(c)=(\alpha_1(c_1),\ldots,\alpha_n(c_n)) \quad \text{and} \quad \beta(d)=(\beta_1(d_1),\ldots,\beta_n(d_n)).
\]
Here $\alpha_i:=\alpha_{\overline{S_i}}$ and $\beta_i:=\beta_{S_i}$ for all $i\in\{1,\ldots,n\}$. Explicitly,
\[
\alpha_i(c_i)=\left\{\begin{aligned} 
-\infty, \ &\text{if} \ c_i<a_i;\\
c_i, \ &\text{if} \ a_i\leq c_i \leq b_i;\\
b_i, \ &\text{if} \ c_i >b_i
\end{aligned}\right.
\quad \text{and} \quad
\beta_i(d_i)=\left\{\begin{aligned} 
a_i, \ &\text{if} \ d_i=-\infty;\\
d_i, \ &\text{otherwise}
\end{aligned}\right.
\]
for every $i\in\{1,\ldots,n\}$.
The next proposition shows us that the composition $\beta\circ \alpha$ is an extension of the convex projection $\pi$ from $\mathbb Z^n$ to $\overline{\mathbb Z}^n$.

\begin{proposition}\label{alfabeta}
Let $\pi\colon \mathbb Z^n\rightarrow [a,b]$ be the convex projection. Then for any $c:=(c_1,\ldots,c_n)\in \mathbb Z^n$,
\[
\pi(c) = (\beta \circ \alpha)(c).
\]
\end{proposition}

\begin{proof}
Suppose first that $n=1$. Recall that $\pi(c)=\max(a,\min(c,b))$. Now there are three cases:
\begin{itemize}
\item If $c\in [a,b]$, then $(\beta \circ \alpha)(c)= \beta(c) = c = \pi(c)$;
\item If $c<a$, then $(\beta\circ\alpha)(c)=\beta(-\infty)=a = \pi(c)$;
\item If $c>b$, then $(\beta\circ\alpha)(c)=\beta(b) = b = \pi(c)$.
\end{itemize}

Suppose next that $n>1$. Using Proposition \ref{alfabeta_lemma}, we may write 
\[
\alpha(c) = (\alpha_1(c_1),\ldots,\alpha_n(c_n)) \quad \text{and} \quad \beta(d) = (\beta_1(d_1),\ldots,\beta_n(d_n))
\]
for all $d\in \overline{[a,b]}$. Similarly, recall that
\[
\pi(c)=(\pi_1(c_1),\ldots,\pi_n(c_n)).
\]
It now follows from the case $n=1$ that
\begin{align*}
(\beta\circ \alpha)(c) &=\beta(\alpha_1(c_1),\ldots,\alpha_n(c_n))\\
&= ((\beta_1\circ\alpha_1)(c_1),\ldots,(\beta_n\circ\alpha_n)(c_n))\\
&=(\pi_1(c_1),\ldots,\pi_n(c_n))\\
&=\pi(c).
\end{align*}
\end{proof}

\begin{remark}\label{beta_huomautus}
In an effort to keep the notation simpler, we only defined $\beta$ for the elements in the image of $\alpha$. Of course, we could have defined $\beta$ in a fully dual fashion to $\alpha$, starting from posets that are strongly bounded from below, adding the point $\infty$ to $\mathbb Z$, and defining a set $\underline{S}$ dually to $\overline{S}$. This would have resulted in the situation where 
\[
(\beta|_{\overline{S}}\circ \alpha)(c)=(\alpha|_{\underline{S}}\circ \beta)(c) = \pi(c)
\]
 for all $c\in \mathbb Z^n$. In other words, the same result would have been achieved.
\end{remark}

We saw in Remark \ref{aar_maar_kuvailu} that if $M$ is encoded by a closed interval $[a,b]$ with the convex projection $\pi\colon\mathbb Z^n\rightarrow [a,b]$, we have $M(c)\cong M(\pi(c))$ for all $c\in \mathbb Z^n$. On the other hand, by Theorem \ref{S-det_luo_1}, we have $\overline{M}(\alpha(c))\cong \overline{M}(c)$ for all $c\in \overline{\mathbb Z}^n$ if $M$ is $S$-determined and $S\subseteq \overline{\mathbb Z}^n$ is finite. In preparation for the proof of Theorem \ref{aar_maar2}, we will now show that a similar result applies to $\beta$ in both cases.

\begin{proposition}\label{apuraja}
Set $u:=(1,1,\ldots,1)\in \mathbb Z^n$. Let $M$ be an $R\mathbb Z^n$-module, and let $c\in\overline{[a,b]}$.
\begin{itemize}
\item[1)] If $M$ has an encoding by the closed interval $[a,b]$ with the convex projection $\pi\colon \mathbb Z^n\rightarrow [a,b]$, then $\overline{M}(c)\cong M(\beta(c))$.
\item[2)] If $\overline{M}$ is $\overline{[a+u,b]}$-determined, then $\overline{M}(c)\cong M(\beta(c))$.
\end{itemize}
\end{proposition}

\begin{proof}
To show 1), suppose that $M$ has an encoding by the closed interval $[a,b]$ with the convex projection $\pi\colon \mathbb Z^n\rightarrow [a,b]$. Then, by the definition of $\overline{M}$,
\[
\overline{M}(c)=\lim_{d\geq c, \ d\in \mathbb Z^n}M(d).
\]
The encoding gives us $M(d)\cong M(\pi(d))$ for all $d\in\mathbb Z^n$. This implies that 
\[
\overline{M}(c)\cong\lim_{d\geq c, \ d\in \mathbb Z^n}M(\pi(d)).
\]
We may now apply Corollary \ref{jatkuvuus_koro} to see that $\overline{M}(c)\cong M(\beta(\alpha(c)))$. Note that $c\in\overline{[a,b]}$ implies $\alpha(c)=c$. Thus $\overline{M}(c)\cong M(\beta(c))$.

Next, to prove 2), let $\overline{M}$ be $\overline{[a+u,b]}$-determined. Since $c\leq \beta(c)$, it is then enough to show that $\overline{[a+u,b]}\cap{\downarrow}c=\overline{[a+u,b]}\cap{\downarrow}\beta(c)$. We instantly have ${\downarrow}c \subseteq {\downarrow}\beta(c)$. For the other direction, let $d:=(d_1,\ldots,d_n)\in \overline{[a+u,b]}\cap{\downarrow}\beta(c)$. We want to show that $d\leq c$. Recall that we may write $\beta(c)=(\beta_1(c_1),\ldots,\beta_n(c_n))$, where 
\[
\beta_i(c_i)=\left\{\begin{aligned} 
&a_i, \ \text{if} \ c_i=-\infty;\\
&c_i, \ \text{otherwise}.
\end{aligned}\right.
\]
for all $i\in \{1,\ldots,n\}$. Suppose that $i \in \{1,\ldots,n\}$. If $\beta_i(c_i)=c_i$, we have $d_i\leq \beta_i(c_i)=c_i$. Otherwise, if $\beta_i(c_i)=a_i$, we must have $d_i=c_i=-\infty$, because $d_i,c_i\in \overline{[a_i+1,b_i]}$. We conclude that $d\leq c$.
\end{proof}

\begin{remark}\label{aarellisyydet}
Let $M$ be a pointwise finitely presented $R\mathbb Z^n$-module and let $c\in\overline{\mathbb Z}^n$. If $M$ is finitely determined with the convex projection $\pi\colon \mathbb Z^n\rightarrow [a,b]$, then from the proof of Proposition \ref{apuraja}, we have $\overline{M}(c) \cong M((\beta\circ\alpha)(c))$, so that $\overline{M}$ is pointwise finitely presented.
\end{remark}

We are now ready to state 

\begin{theorem}\label{aar_maar2}
Let $M$ be a pointwise finitely presented $R\mathbb Z^n$-module. Then the following are equivalent:
\begin{itemize}
\item[1)] $M$ is finitely determined;
\item[2)] $\overline{M}$ is $S$-determined for some finite $S\subseteq \overline{\mathbb Z}^n$;
\item[3)] $\overline{M}$ is finitely presented.
\end{itemize}
\end{theorem}

\begin{proof}
We will first show the equivalence of 1) and 2). Note that for any finite subset $S\subseteq \overline{\mathbb Z}^n$, we can always find $a,b\in \mathbb Z^n$ such that $S\subseteq \overline{[a+u,b]}$. Consider the functor
\[
\alpha'=\alpha_{\overline{[a+u,b]}}\colon\overline{\mathbb Z}^n\rightarrow \overline{[a+u,b]},
\]
and denote its restriction to $\mathbb Z^n$ by $\overline{\alpha}$.  By Theorem \ref{S-det_luo_1}, $\overline{M}$ is $\overline{[a+u,b]}$-determined if and only if $\overline{M}$ is encoded by $\alpha'$. That is, $\overline{M}\cong \res_{\alpha'}N$ for some $R\overline{[a+u,b]}$-module $N$. By restricting to $\mathbb Z^n$, we see that
\[
M\cong \res_{\mathbb Z^n}\res_{\alpha'}N=\res_{\overline{\alpha}}N,
\]
so $\overline{\alpha}$ encodes $M$. Conversely, if $M$ is encoded by $\overline{\alpha}$, then $\overline{M}$ has an obvious encoding by $\alpha'$, because $\overline{\alpha}$ is a surjection on objects. Next, we note that the restriction of $\beta$ to $\overline{[a+u,b]}$,
\[
\overline{\beta}\colon\overline{[a+u,b]}\rightarrow [a,b].
\]
is an isomorphism of posets. Therefore $\overline{\beta}\circ\overline{\alpha}$ is an encoding of $M$ if and only if $\overline{\alpha}$ is an encoding of $M$. These conditions are equivalent to $M$ being finitely determined, because $\overline{\beta}\circ\overline{\alpha}=\pi$. Namely, for all $c\in \mathbb Z^n$, we have $(\overline{\beta}\circ\overline{\alpha})(c)=(\beta\circ\alpha')(c)$, where
\begin{align*}
(\beta\circ\alpha')(c)_i
&=\left\{
\begin{aligned}
&\beta_i(-\infty), \ \text{if} \ c_i=a_i,\\
&\beta_i(\alpha_i(c_i)), \text{else.}
\end{aligned}
\right.\\
&=\left\{
\begin{aligned}
&a_i, \ \text{if} \ c_i=a_i,\\
&\pi_i(c_i), \text{else.}
\end{aligned}
\right.\\
&=\pi(c)_i
\end{align*}
for all $i\in\{1,\ldots,n\}$.

Finally, we observe that the equivalence of 2) and 3) follows from the main result of our previous paper, \cite{Hyry}*{p.~25, Thm.~4.15}. For $\overline{\mathbb Z}^n$-modules, it states that being pointwise finitely presented and $S$-determined for some finite $S\subseteq \overline{\mathbb Z}^n$ is equivalent to being finitely presented. Also note Remark \ref{aarellisyydet}, which shows us that $\overline{M}$ is pointwise finitely presented.
\end{proof}

We are now able to give a ``sharpened'' version of Proposition \ref{aar_maar}.

\begin{corollary}\label{aar_maar2_huom}
If $M$ is an $R\mathbb Z^n$-module and $a,b\in \mathbb Z^n$ such that $a\leq b$, then the following are equivalent:
\begin{itemize}
\item[1)] $M$ is encoded by the convex projection $\pi\colon \mathbb Z^n\rightarrow [a,b]$;
\item[2)] $\overline{M}$ is $\overline{[a+u,b]}$-determined, where $u:=(1,\ldots,1)\in \mathbb Z^n$.
\end{itemize}
\end{corollary}
\begin{proof}
We showed in the proof of Theorem \ref{aar_maar2} that 2) implies 1). Conversely, suppose that 1) holds. Let $c\leq d$ in $\overline{\mathbb Z}^n$ such that $\overline{[a+u,b]}\cap{\downarrow}c=\overline{[a+u,b]}\cap{\downarrow}d$. Coordinatewise, for $i=\{1,\ldots,n\}$, this implies that either $c_i=d_i$, $b_i \leq c_i < d_i$ or $c_i<d_i\leq a_i$. In any case, $(\beta\circ\alpha)(c)=(\beta\circ\alpha)(d)$, so that
\[
\overline{M}(c)\cong M((\beta\circ\alpha)(c))=M((\beta\circ\alpha)(d))\cong \overline{M}(d).
\]
Thus $M(c\leq d)$ is an isomorphism, and $M$ is $\overline{[a+u,b]}$-determined.
\end{proof}

To demonstrate Theorem \ref{aar_maar2} and Corollary \ref{aar_maar2_huom}, it is convenient to take the point of view of topological data analysis, and consider the births and deaths of elements of a module. Given an $R\mathbb Z^n$ module $M$, one can track how an element $x\in \overline{M}(c)$, where $c\in \overline{\mathbb Z}^n$, evolves when mapped with the homomorphisms $M(c\leq c')$, ($c,c' \in \overline{\mathbb Z}^n$). We say that the element $x$ is born at $c$ if it is not in the image of any morphism $M(c'\leq c)$, where $c'<c$. On the other hand, the element $x$ dies at $c''$ if $M(c\leq c'')(x)=0$, but $M(c\leq c')(m)\neq 0$ for all $c\leq c'<c''$. 

Consider now an $R\mathbb Z^2$-module $M$ that is finitely determined, and let $\pi\colon \mathbb Z^2\rightarrow [a,b]$ be the accompanying convex projection. Note that no new elements are born or die in the leftmost edge or the bottom edge of the box $[a,b]$. This follows from the fact that every element on these two edges has already appeared infinite times before, and was born at some infinitary point. Let us write $a=(a_1,a_2)$. For example, if an element, say $x\in M((a_1,c))$, maps to zero on the leftmost edge of $[a,b]$, in $M((a_1,c+1))$, then $x\in \overline{M}((-\infty,c))$ will also map to zero in $\overline{M}((-\infty,c+1))$. Thus $x$ does not ``die'' at the point $(a_1,c+1)$, but rather at the infinitary point $(-\infty,c+1)\in \overline{[a+u,b]}$.

\begin{remark}\label{syntymat_kuolemat}
Let $M$ be an $R \mathbb Z^n$-module and $c\in \overline{\mathbb Z}^n$. Consider the natural homomorphism
\[
\lambda_{\overline{M},c}\colon \underset{d\leq c, d\in \overline{[a+u,b]}}\colim \overline{M}(d)\rightarrow \overline{M}(c).
\]
Following \cite{Hyry}*{p.~15, Def.~3.6}, we say that $c$ is a \emph{birth} if $\lambda_{\overline{M},c}$ is a non-epimorphism, and a \emph{death} if $\lambda_{\overline{M},c}$ is a non-monomorpism. Furthermore, suppose that  $\overline{M}$ is $\overline{[a+u,b]}$-determined, and the births are ``well-behaved'' enough. That is, for any birth $c$, the module $\overline{M}(c)/\Img{\lambda_{\overline{M},c}}$ is projective. The latter of course holds if $R$ is a field. Then, as we discussed in \cite{Hyry}*{p.~21, Remark~3.27}, births and deaths show the positions of the minimal generators and relations of $M$.
\end{remark}

In the next example, we will demonstrate how, for a finitely determined module $M$, the extension $\overline{M}$ has births and deaths at infinitary points that guarantee the existence of a finite presentation of $\overline{M}$. 

\begin{example}\label{tarkea esim}
Let $M$ be an $R\mathbb Z^2$-module that is defined on objects by
\[
M(c)=\left\{\begin{aligned}
&R, \ \text{if} \ c\leq (0,0);\\
&0, \ \text{otherwise},
\end{aligned} \right.
\]
for all $c\in \mathbb Z^2$, and where a morphism $R\rightarrow R$ is always $\id_R$. Then $M$ is finitely determined with the convex projection $\pi\colon \mathbb Z^2\rightarrow [(0,0),(1,1)]$. Now, by Remark \ref{aar_maar2_huom}, $\overline{M}$ is $\overline{[(1,1),(1,1)]}$-determined. Here $\overline{[(1,1),(1,1)]}$ is the set
\[
\{(-\infty,-\infty),(1,-\infty),(-\infty,1),(1,1)\}.
\]
In particular, we have $\overline{M}((-\infty,-\infty))=R$, and
\[
\overline{M}((-\infty,1))=\overline{M}((1,-\infty))=\overline{M}((1,1))=0.
\]
Furthermore, by Theorem \ref{aar_maar2}, $\overline{M}$ is now finitely presented. In more concrete terms, we have an exact sequence of $R\overline{\mathbb Z}^2$-modules 
\[
K \rightarrow N \rightarrow \overline{M} \rightarrow 0,
\]
where
\[
N=R[\Mor_{\overline{\mathbb Z}^2}((-\infty,-\infty),-)]
\]
and
\[
K=R[\Mor_{\overline{\mathbb Z}^2}((1,-\infty),-)]\oplus R[\Mor_{\overline{\mathbb Z}^2}((-\infty,1),-)].
\]
Here $(-\infty,-\infty)$ is the only birth of $M$, while $(1,-\infty)$ and $(-\infty,1)$ are the deaths.
\end{example}

\begin{example}\label{generated_presented}
If $k$ is a field, then it is well known that finitely generated $k\mathbb Z^n$-modules are finitely presented. This result, however, does not apply to $k\overline{\mathbb Z}^n$-modules. For a counterexample, consider a $k\mathbb Z^2$-module $M$, where
\[
M((x,y))=\left\{\begin{aligned}
&k, \text{ if } x+y<0;\\
&0, \text{ otherwise}.
\end{aligned}\right.
\]
Clearly $\overline{M}$ is finitely generated with its only birth in $(-\infty,-\infty)$. It is not finitely presented, since the deaths happen at points $(n,-n)$ for all $n\in\mathbb Z$.
\end{example}

Finally, we want to relate Theorem \ref{aar_maar2} to the work of Perling (\cite{Perling}). Recall that a subset $L\subseteq \overline{\mathbb Z}^n$ is a join-sublattice if $\mub(S)\in L$ for every finite subset $S\subseteq L$. Note that this is equivalent to the condition that $L=\hat{L}$. Given a join-sublattice $L\subseteq \overline{\mathbb Z}^n$, following Perling in \cite{Perling}*{pp.~16-19, Ch.~3.1}, we define the zip-functor 
\[
\zip_L\colon  R\mathbb Z^n\text{-}\textbf{Mod} \rightarrow RL\text{-}\textbf{Mod}
\]
and the unzip-functor
\[
\unzip_L\colon  RL\text{-}\textbf{Mod}\rightarrow R\overline{\mathbb Z}^n\text{-}\textbf{Mod}.
\]
Contrary to Perling, we do not assume that $R$ is a field. The zip-functor maps an $R\mathbb Z^n$-module $M$ to the $RL$-module $\res_L\overline{M}$, whereas he unzip-functor maps an $RL$-module $N$ to an $R\overline{\mathbb Z}^n$-module $\unzip_L N$ defined by
\[
(\unzip_L N)(c)=\left\{\begin{aligned} &N(\mub(L\cap {\downarrow}c)), \ \text{if} \ L\cap {\downarrow}c\neq \emptyset; \\
  &0, \ \text{otherwise}
\end{aligned}\right.
\]
for all $c\in\overline{\mathbb Z}^n$. Note that $\supp(\unzip_L N)\subseteq {\uparrow}L$.

\begin{remark}\label{unzip_selitys}
It turns out that $\unzip_L$ is essentially the same thing as $\res_{\alpha}$, when $L$ is finite and $\alpha:=\alpha_L$. There is the slight complication that $\unzip_L$ is defined for $RL$-modules, while $\res_{\alpha}$ is defined for $R\tilde{L}$-modules. We may, however, extend an $RL$-module $N$ to an $R{\tilde{L}}$-module $\tilde{N}$ by setting 
\[
\tilde{N}((-\infty,\ldots,-\infty))=0,
\] 
if $(-\infty,\ldots,-\infty)\notin L$, and $\tilde{N}(c)=N(c)$, otherwise. Having defined the module $\tilde{N}$ in this way, we see that $\unzip_L N \cong \res_{\alpha} \tilde{N}$.
\end{remark}

Given an $R\mathbb Z^n$-module $M$, the join-sublattice $L$ is called \emph{$M$-admissible} in \cite{Perling}*{p.~18, Def.~3.4} if the condition $\overline{M} \cong \unzip_L\zip_L M$ is satisfied. This leads us to the following proposition.

\begin{proposition}\label{PerlingProposition}
Let $M$ be an $R\mathbb Z^n$-module, and $L$ a finite join-sublattice. Then $L$ is $M$-admissible if and only if $\overline{M}$ is $L$-determined.
\end{proposition}

\begin{proof}
Let $c\in \overline{\mathbb Z}^n$. With the earlier notation, we see that
\[
\unzip_L\zip_L M = \unzip_L\res_L\overline{M}\cong \res_{\alpha} \widetilde{\res_L\overline{M}},
\]
where
\[
(\res_{\alpha}\widetilde{\res_L \overline{M}})(c)=\left\{\begin{aligned} &(\res_{\alpha}\res_{\tilde{L}}\overline{M})(c), \ \text{if} \ L\cap {\downarrow}c\neq \emptyset; \\
  &0, \ \text{otherwise.}
\end{aligned}\right.
\]
Assume first that $\overline{M}\cong \unzip_L\zip_L M$. If $L\cap {\downarrow}c =\emptyset$, we have $\overline{M}(c)=0$ by the definition of the functor $\unzip_L$. But in this case $\alpha(c)\leq c$, so that $L\cap {\downarrow}\alpha(c)=\emptyset$. Using the definition of $\unzip_L$ again, we get
\[
(\res_{\alpha}\res_{\tilde{L}}\overline{M})(c)=\overline{M}(\alpha(c))=0.
\]
 On the other hand, if there is an element $d\in L\cap{\downarrow}c$, then, by the above formula, $\overline{M}(c)\cong (\res_{\alpha}\res_{\tilde{L}}\overline{M})(c)$. Thus, 
\[
\overline{M}\cong \res_{\alpha}\res_{\tilde{L}}\overline{M}
\]
and $\supp(\overline{M})\subseteq {\uparrow} L$, so $\overline{M}$ is $L$-determined by Theorem \ref{S-det_luo_1}.

Conversely, suppose that $\overline{M}$ is $L$-determined. By Theorem \ref{S-det_luo_1}, we have $\overline{M}\cong \res_{\alpha}\res_{\tilde{L}}\overline{M}$ and $\supp(\overline{M})\subseteq {\uparrow} L$. The above formula shows us that
\[
(\unzip_L\zip_L M)(c) = (\res_{\alpha}\res_{\tilde{L}}\overline{M})(c)
\]
for all $c\in {\uparrow}L$. If $c\notin {\uparrow}L$, then $c\notin \supp(\overline{M})$, which means that $\overline{M}(c)=0$. In this case, we also have $(\unzip_L\zip_L M)(c)=0$ by the definition of the functor $\unzip_L$. Thus we have an isomorphism
\[
\overline{M}\cong \unzip_L\zip_L M.
\]
\end{proof}


\end{document}